\documentclass[11pt, letterpaper]{article}
\usepackage{fullpage}

\usepackage{amsmath,amsthm,amssymb,amsfonts}
\usepackage{mathtools}
\usepackage{mathrsfs}
\usepackage[shortlabels]{enumitem}
\usepackage{thmtools}

\usepackage{hyperref}
\usepackage[nameinlink,capitalise,noabbrev]{cleveref}
\crefformat{equation}{(#2#1#3)}

\usepackage{xcolor}
\usepackage{graphicx}
\usepackage{float}

\theoremstyle{definition}
\newtheorem{definition}{Definition}

\newtheorem{problem}[definition]{Problem}

\theoremstyle{plain}

\newtheorem{theorem}[definition]{Theorem}
\newtheorem{lemma}[definition]{Lemma}
\newtheorem{proposition}[definition]{Proposition}
\newtheorem{claim}[definition]{Claim}

\newtheorem{corollary}[definition]{Corollary}

\crefname{claim}{Claim}{Claims}
\crefname{lemma}{Lemma}{Lemmas}

\def \ce {\coloneqq}
\def \E {\mathbb{E}}

\renewcommand{\le}{\leqslant}
\renewcommand{\ge}{\geqslant}
\renewcommand{\leq}{\leqslant}

\renewcommand \b[2] {\binom{#1}{#2}}

\def \mF{\mathcal{F}}
\def \G {\mathcal{G}}
\def \mH{\mathcal{H}}

\def \us {\partial^+}
\def \usH {\partial^+(\mH)}

\hypersetup{
  colorlinks=true,
  linkcolor=blue,
  citecolor=blue,
  urlcolor=blue
}

\title{On the maximum density of $r$-graphs in which \\every $(r+1)$-set spans $0$ or $2$ edges
}
\author{
Vishesh Jain\thanks{Department of Mathematics, Statistics and Computer Science,
University of Illinois Chicago, United States. Emails:
\texttt{visheshj@uic.edu}, \texttt{haoranl8@uic.edu}, \texttt{mubayi@uic.edu}.}
\and
Haoran Luo\footnotemark[1]
\and
Dhruv Mubayi\footnotemark[1]
}
\date{}

\begin{document}
\maketitle

\begin{abstract}
 
In 1984, Frankl and F\"uredi asked for the maximum density of an $n$-vertex $r$-graph in which every $(r+1)$-set of vertices spans $0$ or $2$ edges.
They gave a construction with asymptotic density $2^{1-r}$.
We significantly improve this bound by constructing such \(r\)-graphs with density
$
  \Omega(r^{-3}),
$
thereby improving the dependence on \(r\) from exponential to polynomial.
We also obtain lower bounds for the more general problem in which every $(r+1)$-set spans an even number of edges from $\{0,2,\ldots,2k\}$.
\end{abstract}

\section{Introduction} \label{sec::Int}
An $n$-vertex graph in which every set of 3 vertices spans 0 or 2 edges must be complete bipartite, and hence the maximum number of edges in such a graph is  $\lfloor n^2/4 \rfloor$. We consider a generalization of this extremal problem to $r$-uniform hypergraphs (henceforth abbreviated as \emph{$r$-graphs})  posed by Frankl and F\"uredi~\cite{frankl1984exact} in 1984.

\begin{problem}[Problem~1 in~\cite{frankl1984exact}] \label{pro::0or2}
What is the maximum number of edges in an $n$-vertex $r$-graph in which every set of $r+1$ vertices spans $0$ or $2$ edges?
\end{problem}
In the same paper, they considered the case $r=3$ and gave a complete description of such $3$-graphs: each such $3$-graph is either a blow-up of a fixed $3$-graph on $6$ vertices with $10$ edges or is obtained by taking vertices as points on the unit circle and letting the edges be the triples whose convex hull contains the origin.

For general $r$, Frankl and F\"uredi (Construction 3 in~\cite{frankl1984exact}) generalized the latter case of $r=3$: put $n$ points randomly on the unit sphere in $\mathbb{R}^{r-1}$, and let an $r$-set form an edge if its convex hull contains the origin. This construction gives the lower bound $(1+o(1))2^{1-r}\b{n}{r}$ for \cref{pro::0or2}, which was the best known general lower bound before the present work.
For the upper bound, a double-counting argument by de Caen~\cite{caen1983extension} gives $\frac{n}{r^2} \b{n}{r-1} \approx \frac{1}{r} \b{n}{r}$, see Proposition~14 in~\cite{gunderson2017tournaments}.
Before the present work, apart from the exact $r=3$ description, the only improvement over the general Frankl--F\"uredi construction was for the case $r=4$.
Gunderson and Semeraro~\cite{gunderson2017tournaments} report an earlier random-tournament construction of Baber showing that the asymptotically optimal density $1/4$ can be attained. They subsequently proved that for every prime power $q \equiv 3 \pmod 4$, there is a $4$-graph on $q+1$ vertices and $\frac{q+1}{16}\b{q+1}{3}$ edges such that every set of $5$ vertices spans $0$ or $2$ edges, which matches the upper bound mentioned above.

In this paper, we give a simple construction which significantly improves the lower bound for \cref{pro::0or2}. We note that we do not try to optimize the constant $32$ to the best possible. 
\begin{theorem} \label{1overr3}
    For every uniformity $r \ge 2$ and integer $n \ge 4r^2$, there is an $n$-vertex $r$-graph $\mH$ such that every $(r+1)$-set of vertices in $\mH$ spans exactly $0$ or $2$ edges and
    \[
        |\mH| \ge \frac{1}{32r^3}\b{n}{r}.
    \]
\end{theorem}

We remark that \cref{pro::0or2} is closely related to the classical problem in which every $(r+1)$-set spans at most two edges. Equivalently, this is the Tur\'an problem for $H^r_3$, the unique $r$-graph with three edges on $r+1$ vertices, which is also the $(3,2)$-daisy; the broader daisy problem has recently seen a breakthrough, see~\cite{ellis2024daisies}.
For a hypergraph $H$, let $\pi(H)$ be its Tur\'an density.
When $r=3$, $H^3_3$ is $K_4^-$, the $3$-graph obtained from the clique on $4$ vertices by removing one edge.
Frankl and F\"uredi~\cite{frankl1984exact} gave the lower bound $\pi(K_4^-)\ge 2/7$. 
Baber and Talbot~\cite{baber2011hypergraphs} proved the upper bound $\pi(K_4^-)\le 0.2871$ using flag algebras, and this was later slightly improved to $\pi(K_4^-)\le 0.286889$ by Falgas-Ravry and Vaughan~\cite{falgasravry2013applications}.
For general $r$, the argument of de Caen~\cite{caen1983extension} mentioned above gives $\pi(H^r_3)\le 1/r$, which is also the present general upper bound for \cref{pro::0or2}. On the lower-bound side, the original construction of Frankl and F\"uredi~\cite{frankl1984exact} gives the exponential bound $\pi(H^r_3)\ge 2^{1-r}$; Sidorenko~\cite{sidorenko2024turan} recently proved $\pi(H^r_3)\ge r^{-2}$ for every $r$, and in fact $\pi(H^r_3)\ge (1.7215-o(1))r^{-2}$ as $r\to\infty$; Gunderson and Semeraro~\cite{gunderson2025switching} also obtained the lower bounds
$\pi(H^6_3)\ge \frac{9}{64}$,
$\pi(H^7_3)\ge \frac{35}{2^{11}}$, and
$\pi(H^8_3)\ge \frac{315}{2^{14}}$
via an approach termed tournament switching. More recently, Clemen~\cite{clemen2026applications} improved the asymptotic lower bound to
$
    \pi(H^r_3)=\Omega\left(r^{-2}\sqrt{\log r}\right)
$
using sparse hypergraph colorings.

We remark that progress on the upper bound for \cref{pro::0or2} should also be useful for the at-most-two problem. Indeed, the third author~\cite{mubayi2003hypergraphs} used the exact Frankl--F\"uredi result for \cref{pro::0or2} when $r=3$ to improve de Caen's upper bound for $\pi(K_4^-)$. The idea is that, in an $H^r_3$-free $r$-graph, every $(r+1)$-set spans $0$, $1$, or $2$ edges. If a dense $H^r_3$-free $r$-graph exceeds the corresponding $0/2$ threshold on many small vertex sets, then many of those small sets must contain an $(r+1)$-set spanning exactly one edge. These one-edge local configurations can then be inserted into the extension-counting inequality to improve the upper bound. Therefore, an improved upper bound for \cref{pro::0or2} would, by the same supersaturation mechanism, give an improved upper bound for $\pi(H^r_3)$. 

Our methods in Theorem~\ref{1overr3} apply to the following rather natural generalization of Problem~\ref{pro::0or2}.

\begin{definition} 
For fixed $k\ge 2$, let $M_{\le 2k}(n,r)$ be the maximum number of edges in an $n$-vertex $r$-graph in which the number of edges contained in every $(r+1)$-set lies in  $\{0,2,\ldots,2k\}$.
\end{definition}

For every fixed $r$ and $k$, the limit
\[
    m_{\leq 2k}(r)\ce \lim_{n\to\infty}\frac{M_{\le 2k}(n,r)}{\binom nr}
\]
exists by the standard averaging argument of Katona, Nemetz, and Simonovits~\cite{katona1964problem}. Indeed, if $n\ge m$, then averaging an extremal $n$-vertex example over all $m$-vertex subsets gives
\[
    \frac{M_{\le 2k}(m,r)}{\binom mr}
    \ge
    \frac{M_{\le 2k}(n,r)}{\binom nr}.
\]
Thus the sequence $M_{\leq 2k}(n,r)/\binom nr$ is nonincreasing. We give the following lower bounds for $m_{\le 2k}(r)$.

\begin{theorem} \label{thm::even-counts}
    For every fixed integer $k\ge 2$, there is a constant $c_k>0$ such that, for all sufficiently large $r$,
    \[
        m_{\le 2k} (r)
        \ge 
        c_k r^{-1-4/k}(\log r)^{1/k}.
    \]
\end{theorem}

For example, the special case $k=2$ gives the following logarithmic improvement over the order $r^{-3}$ bound in \cref{1overr3}.
\begin{corollary} \label{cor::024-log}
    There is an absolute constant $c>0$ such that, for all sufficiently large $r$,
    \[
        m_{\le 4} (r)
        \ge 
        c r^{-3}\sqrt{\log r}.
    \]
\end{corollary}

The proof of \cref{thm::even-counts} is inspired by Clemen's recent improvement~\cite{clemen2026applications} for the Tur\'an density of $H^r_3$. In both arguments, one first constructs a dense family of cores, then encodes the remaining bad local configurations as edges of an auxiliary sparse hypergraph, and finally passes to a large independent subfamily. In our setting, the auxiliary hypergraph is a linear hypergraph, so a well-known theorem of Duke, Lefmann, and R\"odl~\cite{duke1995uncrowded} is sufficient.

The remaining part of this paper is organized as follows. In \cref{sec::pro}, we give the proof of \cref{1overr3}. In \cref{sec::even-counts}, we prove \cref{thm::even-counts,cor::024-log}. In \cref{sec::Con}, we give some concluding remarks on related problems.

\section[Proof of Theorem~\ref{1overr3}]{Proof of \cref{1overr3}} \label{sec::pro}
The main difficulty in the construction for \cref{1overr3} is that the property that every $(r+1)$-set spans exactly $0$ or $2$ edges imposes strong restrictions on the $r$-graph. Our main observation is the following lemma which enables us to construct a large collection of $r$-graphs that satisfy this property.

For an $(r-1)$-graph $\mH$, its \emph{upper shadow} $\usH$ is the $r$-graph on the same vertex set as $\mH$ whose edges are the $r$-sets containing at least one edge in $\mH$.
\begin{lemma} \label{goodUpperShadow}
    For every $(r-1)$-graph $\mH$ with the property that every $(r+1)$-set of vertices spans at most one edge, $\usH$ has the property that every $(r+1)$-set of vertices spans exactly $0$ or $2$ edges.
\end{lemma}
\begin{proof}
    Let $S$ be an arbitrary $(r+1)$-set of vertices. By assumption, $S$ spans $0$ or $1$ edge in $\mH$. If $S$ spans no edge in $\mH$, then $S$ spans no edge in $\usH$. If $S$ spans one edge in $\mH$, then exactly two $r$-subsets of $S$ contain this edge, and hence $S$ spans exactly two edges in $\usH$.
\end{proof}

We use the following standard terminology. If \(\mH_0\) is an \(r\)-graph
with vertex set \([m]\), then a \emph{blow-up} of \(\mH_0\) is obtained by
replacing each vertex \(i\in[m]\) by a vertex class \(V_i\), and replacing
each edge \(\{i_1,\ldots,i_r\}\in \mH_0\) by all transversal \(r\)-sets
with one vertex in each of \(V_{i_1},\ldots,V_{i_r}\). 
The blow-up is \emph{balanced} if the classes \(V_1,\ldots,V_m\) have sizes differing by at most one.

\begin{lemma} \label{goodBlowUp}
    Let $\mH_0$ be an $r$-graph with the property that every $(r+1)$-set of vertices spans exactly $0$ or $2$ edges. Then every blow-up of $\mH_0$ also has this property.
\end{lemma}
\begin{proof}
    Let $\mH$ be a blow-up of $\mH_0$, and let $\pi$ be the natural projection from $V(\mH)$ to $V(\mH_0)$. For every $(r+1)$-set $S\subseteq V(\mH)$, consider $|\pi(S)|$. If $|\pi(S)|\le r-1$, then $S$ spans no edge in $\mH$. If $|\pi(S)|=r+1$, then all vertices of $S$ lie in distinct classes, and hence
    \[
        |\mH[S]|=|\mH_0[\pi(S)]|\in\{0,2\}.
    \]
    Finally, suppose that $|\pi(S)|=r$. Then exactly one class contributes two vertices of $S$, and the other $r-1$ classes contribute one vertex each. If $\pi(S)$ is not an edge of $\mH_0$, then $S$ spans no edge in $\mH$. If $\pi(S)$ is an edge of $\mH_0$, then the only edges of $\mH$ inside $S$ are obtained by choosing one of the two vertices in the repeated class together with the unique vertex from each of the other $r-1$ classes, so $S$ spans exactly two edges.
\end{proof}

\begin{lemma} \label{sizeBlowUp}
    Let \(N,m,r\) be integers such that \(r\ge 2\), \(m\ge 2r^2\), and
    \(m\mid N\). Let \(\alpha\) be a positive real number. Suppose
    \(\mH_0\) is an \(r\)-graph with \(m\) vertices and
    \(\alpha\binom mr\) edges. Then there is an \(N\)-vertex blow-up
    \(\mH_N\) of \(\mH_0\) such that
    \[
        |\mH_N|\ge \frac{\alpha}{2}\binom Nr.
    \]
\end{lemma}

\begin{proof}
    Replace each vertex of \(\mH_0\) by a set of size \(N/m\), where distinct vertices are replaced by disjoint sets. Then
    \[
        |\mH_N|
        =
        \alpha\binom mr\left(\frac Nm\right)^r.
    \]
    Since \(m\ge 2r^2\), we have
    \[
        \binom mr
        =
        \frac{m^r}{r!}\prod_{i=0}^{r-1}\left(1-\frac{i}{m}\right)
        \ge
        \frac{m^r}{r!}
        \left(1-\sum_{i=0}^{r-1}\frac{i}{m}\right)
        \ge
        \frac34\frac{m^r}{r!}.
    \]
    Therefore
    \[
        |\mH_N|
        \ge
        \frac{3\alpha}{4}\frac{N^r}{r!}
        \ge
        \frac{\alpha}{2}\binom Nr,
    \]
    as claimed.
\end{proof}

\begin{lemma} \label{averagingDown}
    Let \(N\ge n\ge r\), and let \(\beta>0\). Suppose that \(\mH_N\) is an
    \(N\)-vertex \(r\)-graph such that every \((r+1)\)-set spans exactly
    \(0\) or \(2\) edges and
    \[
        |\mH_N|\ge \beta\binom Nr.
    \]
    Then there is an \(n\)-vertex \(r\)-graph \(\mH\) such that every
    \((r+1)\)-set spans exactly \(0\) or \(2\) edges and
    \[
        |\mH|\ge \beta\binom nr.
    \]
\end{lemma}

\begin{proof}
    Choose an \(n\)-vertex subset \(U\subseteq V(\mH_N)\) uniformly at
    random. Each edge of \(\mH_N\) is contained in \(U\) with probability
    \(\binom nr/\binom Nr\). Hence
    \[
        \mathbb E|\mH_N[U]|
        =
        |\mH_N|\frac{\binom nr}{\binom Nr}
        \ge
        \beta\binom nr.
    \]
    Therefore some choice of \(U\) satisfies
    \[
        |\mH_N[U]|\ge \beta\binom nr.
    \]
    Since the property that every \((r+1)\)-set spans exactly \(0\) or \(2\)
    edges is inherited by induced subgraphs, \(\mH=\mH_N[U]\) has the
    desired local property.
\end{proof}
We also need the following construction, for which we give two proofs.
\begin{proposition} \label{goodr-1}
    For every uniformity $r \ge 2$ and integer $n \ge r$, there is an $n$-vertex $(r-1)$-graph $\mH$ with the property that every $(r+1)$-set of vertices spans at most one edge and
    \[
        |\mH| \ge \frac{1}{4n^2} \b{n}{r-1}.
    \]
\end{proposition}
We remark that, in the range $n\ge 2r$,  \cref{goodr-1} is best possible up to the absolute constant. Indeed, suppose that $r\ge 3$, $n\ge 2r$, and that an $(r-1)$-graph $\mF$ on $n$ vertices has the property that every $(r+1)$-set spans at most one edge. If two distinct edges of $\mF$ contain the same $(r-3)$-set, then their union would have size at most $r+1$, and hence some $(r+1)$-set would contain both of them, a contradiction. Therefore every $(r-3)$-set is contained in at most one edge of $\mF$, and so
\[
    |\mF|\binom{r-1}{2}\le \binom{n}{r-3}.
\]
Equivalently,
\[
    |\mF|\le \frac{2}{(n-r+2)(n-r+3)}\binom{n}{r-1}.
\]
In particular, if $n\ge 2r$, then $|\mF|\le \frac{8}{n^2}\binom{n}{r-1}$.

\begin{proof} [First proof of \cref{goodr-1}]
    We use the following theorem of Graham and Sloane~\cite{graham1980lower}.

    Let $A(n,2\delta,w)$ be the maximum number of codewords in any binary code of length $n$, constant weight $w$, and minimum Hamming distance at least $2\delta$.

\begin{theorem}[Theorem~4 in~\cite{graham1980lower}] \label{grahamCode}
    Let $q$ be a prime power such that $q \ge n$. Then,
    \[
        A(n,2\delta,w) \ge \frac{1}{q^{\delta-1}} \b{n}{w}.
    \]
\end{theorem}
    Now let $\delta = 3$, $w = r-1$, and let $q$ be a prime number in $[n,2n]$. By \cref{grahamCode}, we have
    \[
        A(n,6,r-1) \ge \frac{1}{q^{3-1}} \b{n}{r-1} \ge \frac{1}{4n^2} \b{n}{r-1}.
    \]
    Viewing the codewords as the indicator vectors of the $(r-1)$-edges, we get an $(r-1)$-graph $\mH$. Suppose for a contradiction that there is an $(r+1)$-set $S$ spanning at least two edges, say $e_1,e_2$. Since $|e_1|=|e_2|=r-1$ and $e_1,e_2\subseteq S$, we have $|e_1\cap e_2|\ge r-3$. Hence, the corresponding codewords for $e_1,e_2$ have Hamming distance at most $4$, a contradiction.
\end{proof}

\begin{proof}[Second proof of \cref{goodr-1}]
    For $r=2$, taking a single vertex as the only $1$-edge gives the desired construction. Hence, we may assume $r\ge 3$.

    Let $q$ be a prime number in $[n,2n]$, and identify the vertex set with an arbitrary $n$-element subset $U\subseteq \mathbb F_q$. Define
    \[
        \lambda:\mathbb F_q\to \mathbb F_q^2,\qquad
        \lambda(t)=(t,t^2).
    \]
    Note that if
    \[
        \lambda(a)+\lambda(b)=\lambda(c)+\lambda(d),
    \]
    then $a+b=c+d$ and $a^2+b^2=c^2+d^2$. Since $q$ is odd, this implies $ab=cd$, and then $\{a,b\}$ and $\{c,d\}$ have the same sum and product, so they are equal.

    For each $z\in \mathbb F_q^2$, define
    \[
        \G_z=\left\{A\in \binom{U}{r-1}: \sum_{a\in A}\lambda(a)=z\right\}.
    \]
    The families $\G_z$ partition $\binom{U}{r-1}$ as $z$ ranges over $\mathbb F_q^2$. Hence, by averaging, there exists $z\in \mathbb F_q^2$ such that
    \[
        |\G_z|\ge \frac{1}{q^2}\binom{n}{r-1}
        \ge \frac{1}{4n^2}\binom{n}{r-1}.
    \]
    We claim that this $\G_z$ has the required local property.

    Let $S\in \binom{U}{r+1}$. Suppose that $A,B\in \G_z$ and $A,B\subseteq S$. Write
    \[
        S\setminus A=\{u,v\},\qquad S\setminus B=\{u',v'\}.
    \]
    Since $\sum_{a\in A}\lambda(a)=\sum_{b\in B}\lambda(b)=z$, we have
    \[
        \lambda(u)+\lambda(v)
        =
        \sum_{s\in S}\lambda(s)-z
        =
        \lambda(u')+\lambda(v').
    \]
    By the Sidon property of $\lambda$, we get $\{u,v\}=\{u',v'\}$, and hence $A=B$. Therefore, every $(r+1)$-set contains at most one member of $\G_z$, as claimed.
\end{proof}

\begin{proof}[Proof of \cref{1overr3}]
    By \cref{goodr-1}, there exists an $(r-1)$-graph $\mH^{(r-1)}$ with $2r^2$ vertices and at least $\frac{1}{16r^4}\b{2r^2}{r-1}$ edges such that every $(r+1)$-set of vertices spans at most one edge. Let $\mH_0 = \us(\mH^{(r-1)})$. By \cref{goodUpperShadow}, every $(r+1)$-set of vertices in $\mH_0$ spans exactly $0$ or $2$ edges.

    \begin{claim}
        $|\mH_0| \ge \frac{1}{16r^3} \b{2r^2}{r}$.
    \end{claim}
        \begin{proof}
        We first note that every $r$-set contains at most one edge of
        $\mH^{(r-1)}$. Indeed, if an $r$-set $S$ contained two distinct
        edges $A,B\in \mH^{(r-1)}$, then, since $\mH^{(r-1)}$ has
        $2r^2\ge r+1$ vertices, we could choose a vertex outside $S$ and
        extend $S$ to an $(r+1)$-set containing both $A$ and $B$, which would contradict the defining property of $\mH^{(r-1)}$.

        Therefore, each edge of $\mH_0=\us(\mH^{(r-1)})$ contains
        a unique edge of $\mH^{(r-1)}$. Since each $(r-1)$-edge of
        $\mH^{(r-1)}$ is contained in exactly $2r^2-(r-1)$ many $r$-sets,
        we have
        \begin{align*}
            |\mH_0|
            &= (2r^2-(r-1)) |\mH^{(r-1)}| \\
            &\ge \frac{2r^2-r+1}{16r^4}\binom{2r^2}{r-1}
             = \frac{2r^2-r+1}{16r^4}\cdot
               \frac{r}{2r^2-r+1}\binom{2r^2}{r} \\
            &= \frac{1}{16r^3}\binom{2r^2}{r}. \qedhere
        \end{align*}
    \end{proof}
        Choose an integer \(N\ge n\) divisible by \(2r^2\). Applying
    \cref{sizeBlowUp} with \(m=2r^2\) and \(\alpha=1/(16r^3)\), we obtain
    an \(N\)-vertex blow-up \(\mH_N\) of \(\mH_0\) such that
    \[
        |\mH_N|\ge \frac{1}{32r^3}\binom Nr.
    \]
    By \cref{goodBlowUp}, every \((r+1)\)-set of vertices in \(\mH_N\)
    spans exactly \(0\) or \(2\) edges. Applying \cref{averagingDown} with
    \(\beta=1/(32r^3)\), we obtain an \(n\)-vertex \(r\)-graph \(\mH\) such
    that every \((r+1)\)-set spans exactly \(0\) or \(2\) edges and
    \[
        |\mH|\ge \frac{1}{32r^3}\binom nr. \qedhere
    \]

\end{proof}

\section{Proofs for relaxed even local counts} \label{sec::even-counts}

We prove \cref{thm::even-counts}; \cref{cor::024-log} then follows by taking $k=2$. Throughout this section \(k\ge 2\) is fixed, and \(c_k,C_k>0\) denote constants depending only on \(k\), whose values may change from line to line.
We assume that \(r\) is sufficiently large in terms of \(k\); in particular, \(r+1\ge 2k+4\).
Let \(q\) be an odd prime satisfying
\(
    r^2\le q\le 2r^2,
\)
which exists by Bertrand's postulate. 
For every $c\in\mathbb F_q$, define
\[
    \G_c=\left\{A\in\binom{\mathbb F_q}{r-1}:\sum_{a\in A}a=c\right\}.
\]
By averaging over $c\in\mathbb F_q$, we may fix $c$ such that
\begin{equation} \label{eq::Gc-size}
    |\G_c|\ge \frac{1}{q}\binom{q}{r-1}.
\end{equation}

\begin{lemma} \label{lem::core-local}
The family \(\G_c\) has the following two properties.
\begin{itemize}
\item If \(S\in\binom{\mathbb F_q}{r+1}\), then the members of \(\G_c\) contained in \(S\) are in bijection with the pairs \(P\in\binom S2\) satisfying
\[
    \sum_{u\in P}u=\sum_{s\in S}s-c,
\]
via \(A=S\setminus P\). In particular, the corresponding pairs are pairwise disjoint.
\item Every \(r\)-set contains at most one member of \(\G_c\).
\end{itemize}
\end{lemma}

\begin{proof}
The first assertion follows directly from
\[
    \sum_{a\in S\setminus P}a=c
    \qquad\Longleftrightarrow\qquad
    \sum_{u\in P}u=\sum_{s\in S}s-c.
\]
Since \(q\) is odd, the two-element subsets of \(\mathbb F_q\) with a fixed sum form a matching; hence the corresponding pairs in \(S\) are pairwise disjoint.

For the second assertion, suppose that \(X\in\binom{\mathbb F_q}{r}\) and
\(X\setminus\{x\},X\setminus\{y\}\in\G_c\). Then
\[
    \sum_{z\in X}z-x=c=\sum_{z\in X}z-y,
\]
so \(x=y\).
\end{proof}

For \(s\ge 1\), call a collection of \(s\) distinct members of \(\G_c\) \emph{bad} if all its members are contained in a common \((r+1)\)-set.
Let \(B_s\) denote the number of bad \(s\)-collections.

\begin{lemma} \label{cl::bad-s}
For each \(s\in\{k+1,k+2\}\),
\[
    B_s\le C_k r^4|\G_c|.
\]
\end{lemma}

\begin{proof}
Fix \(s\in\{k+1,k+2\}\), and suppose that \(A_1,\ldots,A_s\in\G_c\) are distinct and contained in some \(S\in\binom{\mathbb F_q}{r+1}\).
Write \(P_i=S\setminus A_i\).
By \cref{lem::core-local}, the pairs \(P_i\) are pairwise disjoint and have a common sum, say
\[
    \sum_{x\in P_i}x=L
    \qquad\text{for all }i\in[s].
\]
Let \(R=S\setminus(P_1\cup\cdots\cup P_s)\). Then \(|R|=r+1-2s\). Since
\[
    \sum_{x\in S}x=\sum_{x\in R}x+sL
    \qquad\text{and}\qquad
    c=\sum_{x\in A_i}x=\sum_{x\in S}x-L,
\]
we have
\[
    (s-1)L=c-\sum_{x\in R}x.
\]
For sufficiently large \(r\), \(s-1\ne 0\) in \(\mathbb F_q\), so \(R\) determines \(L\).
Once \(R\) and \(L\) are fixed, the pairs \(P_1,\ldots,P_s\) form an
\(s\)-element submatching of the matching of pairs with sum \(L\). This
matching has at most \(q/2\) edges, and hence the number of possible
unordered collections \(\{P_1,\ldots,P_s\}\) is at most
\[
    \binom{\lfloor q/2\rfloor}{s}\le q^s.
\]
Therefore
\[
    B_s\le \binom{q}{r+1-2s}q^s.
\]
Using \eqref{eq::Gc-size} and \(q-r\ge q/2\), we get
\[
    \frac{B_s}{|\G_c|}
    \le
    q\frac{\binom{q}{r+1-2s}}{\binom{q}{r-1}}q^s
    \le
    C_k q^{s+1}\left(\frac rq\right)^{2s-2}
    =
    C_k r^{2s-2}q^{3-s}
    \le C_k r^4,
\]
where the last inequality uses \(s\ge k+1\ge 3\) and \(q\ge r^2\).
\end{proof}

\begin{lemma} \label{lem::sparse-cores}
There is a subfamily \(\G_1\subseteq\G_c\) such that
\begin{equation} \label{eq::G1-general-size}
    |\G_1|\ge c_k r^{-4/(k+1)}|\G_c|,
\end{equation}
no \((r+1)\)-set contains more than \(k+1\) members of \(\G_1\), and the number of bad \((k+1)\)-collections in \(\G_1\) is at most \(C_k|\G_c|\).
\end{lemma}

\begin{proof}
Choose a sufficiently small constant \(\alpha=\alpha(k)>0\), and set
\[
    \rho=\alpha r^{-4/(k+1)}.
\]
Let \(\G'\subseteq\G_c\) be obtained by retaining each member independently with probability \(\rho\).
Let \(Y_s\) be the number of bad \(s\)-collections contained in \(\G'\).
By \cref{cl::bad-s},
\[
    \E |\G'|=\rho|\G_c|,\qquad
    \E Y_{k+1}\le C_k\alpha^{k+1}|\G_c|,
    \qquad
    \E Y_{k+2}\le C_k\alpha^{k+1}\rho|\G_c|.
\]
Since \(\rho|\G_c|\to\infty\) as \(r\to\infty\), the Chernoff bound for \(|\G'|\), Markov's inequality for \(Y_{k+1},Y_{k+2}\), and the choice of \(\alpha\) imply that, with positive probability,
\[
    |\G'|\ge \frac12\rho|\G_c|,
    \qquad
    Y_{k+2}\le \frac14\rho|\G_c|,
    \qquad
    Y_{k+1}\le C_k|\G_c|.
\]
Fix such a realization.

For each bad \((k+2)\)-collection in \(\G'\), choose one of its members and delete all members chosen in this way.
Let the resulting family be \(\G_1\).
Then
\[
    |\G_1|\ge |\G'|-Y_{k+2}\ge \frac14\rho|\G_c|,
\]
which gives \eqref{eq::G1-general-size}.
No bad \((k+2)\)-collection remains, so no \((r+1)\)-set contains more than \(k+1\) members of \(\G_1\).
Finally, deleting members cannot create bad \((k+1)\)-collections, so the number of such collections in \(\G_1\) is at most \(Y_{k+1}\le C_k|\G_c|\).
\end{proof}

\begin{lemma} \label{lem::independent-cores}
There is a subfamily \(\G_2\subseteq\G_1\) such that every \((r+1)\)-set contains at most \(k\) members of \(\G_2\), and
\begin{equation} \label{eq::G2-general-size}
    |\G_2|\ge c_k|\G_c|r^{-4/k}(\log r)^{1/k}.
\end{equation}
\end{lemma}

\begin{proof}
Define a \((k+1)\)-uniform hypergraph \(\mathcal T\) on vertex set \(\G_1\) by declaring a \((k+1)\)-set to be an edge if its members are contained in a common \((r+1)\)-set.

We first show that \(\mathcal T\) is linear.
Suppose that two edges share two vertices \(A,B\in\G_1\).
Then \(A\) and \(B\) are contained in a common \((r+1)\)-set, so \(|A\cap B|\ge r-3\).
The case \(|A\cap B|=r-2\) is impossible, since then the \(r\)-set \(A\cup B\) would contain two distinct members of \(\G_c\), contradicting \cref{lem::core-local}.
Thus \(|A\cap B|=r-3\), and \(A\cup B\) is the unique \((r+1)\)-set containing both \(A\) and \(B\).
By \cref{lem::sparse-cores}, this \((r+1)\)-set contains at most \(k+1\) members of \(\G_1\).
Hence there is at most one edge of \(\mathcal T\) containing \(A\) and \(B\), proving linearity.

The number of edges of \(\mathcal T\) is at most the number of bad \((k+1)\)-collections in \(\G_1\), hence at most \(C_k|\G_c|\).
Together with \eqref{eq::G1-general-size}, this shows that the average degree of \(\mathcal T\) is at most
\[
    D=C_k r^{4/(k+1)}.
\]
By the linear-hypergraph form of the theorem of Duke, Lefmann, and R\"odl~\cite{duke1995uncrowded}, every linear \((k+1)\)-uniform hypergraph on \(N\) vertices with average degree at most \(D\) has an independent set of size at least
\[
    c_k N\left(\frac{\log D}{D}\right)^{1/k}.
\]
Applying this to \(\mathcal T\), and using \(D=C_k r^{4/(k+1)}\), gives an independent set \(\G_2\subseteq\G_1\) satisfying
\[
    |\G_2|
    \ge
    c_k|\G_1|\left(r^{-4/(k+1)}\log r\right)^{1/k}
    \ge
    c_k|\G_c|r^{-4/k}(\log r)^{1/k}.
\]
Since \(\G_2\) is independent in \(\mathcal T\), no \((r+1)\)-set contains \(k+1\) members of \(\G_2\).
\end{proof}

\begin{lemma}
\label{lem::upper-shadow-even}
Let
\[
    \mH_0=\us\G_2=\left\{X\in\binom{\mathbb F_q}{r}: \exists A\in\G_2 \text{ such that } A\subseteq X\right\}.
\]
Then every \((r+1)\)-set spans one of \(0,2,\ldots,2k\) edges in \(\mH_0\), and
\begin{equation} \label{eq::shadow-size-even-log}
    |\mH_0|\ge c_k r^{-4/k}(\log r)^{1/k}\frac rq\binom qr.
\end{equation}
\end{lemma}

\begin{proof}
By \cref{lem::core-local}, every \(r\)-set contains at most one member of \(\G_c\).
Hence every edge of \(\mH_0\) contains a unique member of \(\G_2\), and therefore
\[
    |\mH_0|=(q-r+1)|\G_2|.
\]
Combining this with \eqref{eq::Gc-size} and \eqref{eq::G2-general-size} gives
\[
    |\mH_0|
    \ge
    (q-r+1)c_k r^{-4/k}(\log r)^{1/k}\frac1q\binom{q}{r-1}
    =
    c_k r^{-4/k}(\log r)^{1/k}\frac rq\binom qr,
\]
which proves \eqref{eq::shadow-size-even-log}.

Now fix \(S\in\binom{\mathbb F_q}{r+1}\), and suppose that \(S\) contains exactly \(m\) members of \(\G_2\).
By \cref{lem::independent-cores}, \(m\le k\).
Writing these members as \(A_i=S\setminus P_i\), the pairs \(P_i\) are pairwise disjoint by \cref{lem::core-local}.
Each \(A_i\) contributes exactly the two \(r\)-subsets of \(S\) obtained by adding one element of \(P_i\), and these contributions are disjoint because every \(r\)-set contains at most one member of \(\G_c\).
Thus
\[
    |\mH_0\cap\binom Sr|=2m\in\{0,2,\ldots,2k\}. \qedhere
\]
\end{proof}

\begin{proof}[Proof of \cref{thm::even-counts}]
Let \(\mH_0\) be the \(q\)-vertex \(r\)-graph from \cref{lem::upper-shadow-even}.
For each \(N\), take a balanced blow-up of the \(q\) vertices into \(N\) vertices, and include all transversal \(r\)-sets whose labels form an edge of \(\mH_0\).
Call the resulting \(r\)-graph \(\mH_N\).

Let \(\pi\) be the projection to the \(q\) labels, and let \(S\) be an \((r+1)\)-set in the blow-up.
If \(|\pi(S)|\le r-1\), then \(S\) spans no edge.
If \(|\pi(S)|=r\), then exactly one label is repeated, so \(S\) spans either no edge or exactly two edges.
If \(|\pi(S)|=r+1\), then \(S\) is transversal and the number of edges spanned by \(S\) equals
\[
    |\mH_0\cap\binom{\pi(S)}r|\in\{0,2,\ldots,2k\}.
\]
Thus \(\mH_N\) satisfies the required local condition. Hence
\[
    M_{\le 2k}(N,r)\ge |\mH_N|.
\]

It remains to estimate the limiting density as \(N\to\infty\).
For the balanced blow-up,
\[
    \lim_{N\to\infty}\frac{|\mH_N|}{\binom Nr}
    =
    |\mH_0|\frac{r!}{q^r}.
\]
By \eqref{eq::shadow-size-even-log},
\[
    |\mH_0|\frac{r!}{q^r}
    \ge
    c_k r^{-4/k}(\log r)^{1/k}\frac rq
    \prod_{i=0}^{r-1}\left(1-\frac iq\right).
\]
Since \(q\ge r^2\),
\[
    \prod_{i=0}^{r-1}\left(1-\frac iq\right)
    \ge
    1-\sum_{i=0}^{r-1}\frac iq
    \ge \frac12,
\]
and since \(q\le 2r^2\), we also have \(r/q\ge 1/(2r)\).
Therefore
\[
    \lim_{N\to\infty}\frac{M_{\le 2k}(N,r)}{\binom Nr}
    \ge
    \lim_{N\to\infty}\frac{|\mH_N|}{\binom Nr}
    \ge
    c_k r^{-1-4/k}(\log r)^{1/k}. \qedhere
\]
\end{proof}

\section{Concluding remarks} \label{sec::Con}
It is natural to consider the following more general local problem. For a set $\Lambda\subseteq \{0,1,\ldots,r+1\}$, let $M_\Lambda(n,r)$ be the maximum size of an $r$-graph $\mH$ on $n$ vertices such that
$
    |\mH\cap \binom{S}{r}|\in \Lambda
$
for every $(r+1)$-set $S$. For every fixed $r$ and $\Lambda$, the limit
\[
    m_\Lambda(r)\ce \lim_{n\to\infty}\frac{M_\Lambda(n,r)}{\binom nr}
\]
exists by the same averaging argument as before.
In this notation, \cref{pro::0or2} asks for lower and upper bounds on $m_{\{0,2\}}(r)$.

The results in \cref{thm::even-counts,cor::024-log} suggest that the structure of the allowed set $\Lambda$ is important. A basic open direction is to determine which choices of $\Lambda$ allow density of order $r^{-2}$, or even larger, and which choices are constrained to the $r^{-3}$ scale by upper-shadow-type barriers.

\section*{Acknowledgments and declaration on the use of generative AI}

A slightly weaker lower bound for Theorem~\ref{1overr3} of the form $\Omega(r^{-5})$ was obtained by the authors without any use of AI tools, where instead of~\cref{goodr-1}, a simple probabilistic deletion method was used. Subsequently, ChatGPT 5.4 Pro was used to find the better constructions in the proof of \cref{goodr-1} to obtain the $\Omega(r^{-3})$ bound; additionally, upon being given a sketch of the argument, it produced a proof of the bound in \cref{thm::even-counts}. All proofs and calculations were checked by the authors.

The authors are grateful to Huy Tuan Pham for helpful discussions.

V.J.~is partially supported by NSF grant DMS-2237646. D.M.~is partially supported by NSF grant DMS-2153576.

\end{document}